\documentclass[10.5pt,reqno]{amsart}
\usepackage{longtable} 
\usepackage{hyperref}
\usepackage[T1]{fontenc}
\usepackage[utf8]{inputenc}
\usepackage[english]{babel}
\usepackage{textcomp}
\usepackage{dsfont}
\usepackage{latexsym}
\usepackage{amssymb}
\usepackage{amsthm}
\usepackage{amsmath}
\DeclareMathAlphabet{\mathpzc}{OT1}{pzc}{m}{en}
\usepackage{yfonts}
\usepackage{xfrac}
\usepackage{newlfont}
\usepackage{graphicx}
\usepackage{mathtools}
\usepackage{comment}
\usepackage{indentfirst}
\usepackage{braket}
\usepackage{mathrsfs}
\usepackage{xcolor}

\usepackage{etoolbox}

\usepackage{scalerel}[2014/03/10]
\usepackage[usestackEOL]{stackengine}
\newcommand{\dashint}{\,\ThisStyle{\ensurestackMath{%
			\stackinset{c}{.2\LMpt}{c}{.5\LMpt}{\SavedStyle-}{\SavedStyle\phantom{\int}}}%
		\setbox0=\hbox{$\SavedStyle\int\,$}\kern-\wd0}\int}

\DeclareMathOperator{\card}{Card}

\DeclareMathOperator{\pr}{pr}
\DeclareMathOperator{\rk}{rk}

\DeclareMathOperator{\Cl}{Cl}
\DeclareMathOperator{\Tr}{Tr}

\renewcommand{\Re}{\mathrm{Re}\,}
\renewcommand{\Im}{\mathrm{Im}\,}

\newcommand{\ee}{\mathrm{e}}

\newcommand{\Aut}{\mathrm{Aut}}

\newcommand{\loc}{\mathrm{loc}}

\newcommand{\dd}{\mathrm{d}}

\DeclarePairedDelimiter{\abs}{\lvert}{\rvert}

\DeclarePairedDelimiter{\norm}{\lVert}{\rVert}

\let\originalleft\left
\let\originalright\right
\renewcommand{\left}{\mathopen{}\mathclose\bgroup\originalleft}
\renewcommand{\right}{\aftergroup\egroup\originalright}

\newcommand{\N}{\mathds{N}}

\newcommand{\Db}{\mathds{D}}
\newcommand{\C}{\mathds{C}}

\newcommand{\R}{\mathds{R}}

\newcommand{\T}{\mathds{T}}
\newcommand{\bD}{{\mathrm{b}D}}

\newcommand{\Cc}{\mathcal{C}}

\newcommand{\Hc}{\mathcal{H}}

\newcommand{\Pc}{\mathcal{P}}

\newcommand{\meg}{\leqslant}
\newcommand{\Meg}{\geqslant}
\newcommand{\eps}{\varepsilon}
\renewcommand{\phi}{\varphi}
\newcommand{\mi}{\mu}

\title[Clark Measures on Symmetric Domains]{Clark Measures Associated with Rational Inner Functions on Bounded Symmetric Domains}

\begin{document}

\theoremstyle{definition}
\newtheorem{deff}{Definition}[section]

\newtheorem{oss}[deff]{Remark}

\newtheorem{ass}[deff]{Assumptions}

\newtheorem{nott}[deff]{Notation}

\newtheorem{ex}[deff]{Example}

\theoremstyle{plain}
\newtheorem{teo}[deff]{Theorem}

\newtheorem{lem}[deff]{Lemma}

\newtheorem{prop}[deff]{Proposition}

\newtheorem{cor}[deff]{Corollary}

\author[M. Calzi]{Mattia Calzi}

\address{Dipartimento di Matematica, Universit\`a degli Studi di
	Milano, Via C. Saldini 50, 20133 Milano, Italy}
\email{{\tt mattia.calzi@unimi.it}} 

\thanks{{\em Math Subject Classification 2020}: 32M15, 31C10 
}
\thanks{The author is a member of the 	Gruppo Nazionale per l'Analisi
	Matematica, la Probabilit\`a e le	loro Applicazioni (GNAMPA) of
	the Istituto Nazionale di Alta Matematica (INdAM). The author was partially funded by the INdAM-GNAMPA Project CUP\_E53C22001930001.
} 

\begin{abstract}
	Given a bounded symmetric domain $D$ in $\C^n$, we consider the Clark measures $\mi_\alpha$, $\alpha\in \T$, associated with a rational inner function $\phi$ from $D$ into the unit disc   in $\C$. We show that $\mi_\alpha=c\abs{\nabla \phi}^{-1}\chi_{\bD \cap \phi^{-1}(\alpha)}\cdot \Hc^{m-1}$, where $m$ is the dimension of the \v Silov boundary $\bD$ of $D$ and $c$ is a suitable constant.  Denoting with $H^2(\mi_\alpha)$ the closure of the space of holomorphic polynomials in $L^2(\mi_\alpha)$, we characterize the $\alpha$ for which $H^2(\mi_\alpha)=L^2(\mi_\alpha)$ when $D$ is a polydisc; we also provide some necessary and some sufficient conditions for general domains.	
\end{abstract}
\maketitle

\section{Introduction}

Denote with $\Db$  the unit disc in $\C$, and take a holomorphic function $\phi\colon \Db \to \Db$. Then, for every $\alpha$ in the torus $\T$,
\[
\Re\left( \frac{\alpha+ \phi}{\alpha-\phi} \right)=\frac{1-\abs{\phi}^2}{\abs{\alpha-\phi}^2}
\]
is a positive harmonic function. By the Riesz--Herglotz representation theorem, there is a unique positive Radon measure $\mi_\alpha$ on $\T$ such that
\[
\Re\left( \frac{\alpha+ \phi(z)}{\alpha-\phi(z)} \right)=(\Pc \mi_\alpha)(z)=\int_\T \frac{1-\abs{z}^2}{\abs{\alpha'-z}^2}\,\dd \mi_\alpha(\alpha')
\]
for every $z\in \Db$, where $\Pc$ denotes the Poisson integral operator (on $\T$). 
The measures $\mi_\alpha$, $\alpha\in \T$ are called Clark (or Aleksandrov--Clark) measures, and were introduced by D.\ N.\ Clark in~\cite{Clark} in order to study the restricted shift operator. Their main properties were then established by A.\ B.\ Aleksandrov (cf., e.g.,~\cite{Aleksandrov1,Aleksandrov2,Aleksandrov3,Aleksandrov4,Aleksandrov5}).
Among these, $\mi_\alpha$ is singular (with respect to the normalized Haar measure $\beta_\T$ on $\T$) if and only if $\phi$ is inner, that is, $\phi$ has radial limits of modulus $1$ at almost every point in $\T$. In addition, 
\[
\int_\T \mi_\alpha\,\dd \beta_\T(\alpha)=\beta_\T
\]
in the sense of integrals of positive measures. This is known as Aleksandrov's disintegration theorem, since its first formulation was only concerned with inner functions, in which case the above integral of measures is indeed a disintegration of $\beta_\T$ relative to its pseudo-image measure $\beta_\T$ under $\phi$ (more precisely, under the boundary value function associated with $\phi$). Cf., e.g.,~\cite{CimaMathesonRoss,PoltoratskiSarason} for an account of the classical theory of Clark measures.

Clark measures proved to be a valuable tool in the study of both the theory of holomorphic functions on $\Db$ and in the study of linear contractions.  Consequently, various extensions to more general domains have been introduced in the last decades. On the one hand, inspired by the link with the study of contractions, M.\ T.\ Jury and then R.\ T.\ W.\ Martin (cf.~\cite{Jury1,Jury2,JuryMartin}) developed a theory of operator-valued Clark measures which may be used to study row contractions. On the other hand, A.\ B.\ Aleksandrov and E.\ Doubtsov (cf.~\cite{Doubtsov,AleksandrovDoubtsov,AleksandrovDoubtsov2})  introduced a notion of Clark measures  associated with $\Db$-valued functions (instead of operator-valued functions as in the previous case) on polydiscs and on Euclidean unit balls. 
More recently, the author proposed an extension of this notion of Clark measures to $\Db$-valued functions on general bounded  symmetric domains, as well as in symmetric Siegel domains (cf.~\cite{BoundedClark,SiegelClark}). 
Recall that a (circular) bounded symmetric domain  $D$ is a bounded circular convex open subset of $\C^n$ such that every $z\in D$ is the unique fixed point of a holomorphic involution of $D$. Given a holomorphic function $\phi\colon D\to \Db$, the definition of the Clark measure $\mi_\alpha$ (according to~\cite{BoundedClark}) is then formally similar to the one in the case $D=\Db$. In fact, $\mi_\alpha$ is the unique positive Radon measure on the \v Silov boundary $\bD$ of $D$ (that is, the set of the extremal points of $D$) such that
\[
\Pc \mi_\alpha= \Re\left( \frac{\alpha+ \phi}{\alpha-\phi} \right)=\frac{1-\abs{\phi}^2}{\abs{\alpha-\phi}^2}
\]
on $D$, where $\Pc \mi_\alpha$ denotes the Poisson--Szeg\H o integral of $\mi_\alpha$, which is an higher-dimensional analogue of the usual Poisson integral. For example, if $D=\Db^n$, then $\bD=\T^n$ and
\begin{equation}\label{eq:3}
(\Pc \mi_\alpha)(z_1,\dots, z_n)=\int_{\T^n} \prod_{j=1}^n \frac{1-\abs{z_j}^2}{\abs{\alpha_j-z_j}^2}\,\dd \mi_\alpha(\alpha_1,\dots,\alpha_n)
\end{equation}
for every $(z_1,\dots, z_n)\in \Db^n$.

Even though these more general Clark measures enjoy several of the basic properties of the classical ones, some important results do not seem to hold in full generality. For example, when $D=\Db$, it is known that  $H^2(\mi_\alpha)$, that is, the closure of the space of holomorphic polynomials in $L^2(\mi_\alpha)$, equals $L^2(\mi_\alpha)$ if and only if $\phi$ is extremal in the unit ball of $H^\infty(\Db)$; in particular, this happens if $\phi$ is inner. 
On the one hand, when $D$ is the unit ball in $\C^n$, then $H^2(\mi_\alpha)= L^2(\mi_\alpha)$ for $\phi$ inner (cf.~\cite[Lemma 4.2]{AleksandrovDoubtsov}).\footnote{A holomorphic function $\phi\colon D\to \Db$ is inner if its radial boundary values on $\bD$ are unimodular almost everywhere.} On the other hand, this property is known to fail  for (some) inner functions already in the bidisc. 
Notice that the relevance of the equality $H^2(\mi_\alpha)= L^2(\mi_\alpha)$ is also related to the fact that it characterizes the injectivity of a natural surjective  partial isometry (the normalized Cauchy--Szeg\H o transform, at least when $\phi(0)=0$) from $L^2(\mi_\alpha)$ onto some generalized de Branges--Rovnyak space (cf.~\cite[Theorem 4.16]{BoundedClark}). In the classical case, this operator is the link between the theory of Clark measures and the theory of contractions; even though it is still unclear whether such an operator provides a  link to the theory of a reasonable class of operators also in the general case, investigating its properties may still be of some interest.

The study of Clark measures associated with rational inner functions on polydiscs was initiated in~\cite{BCS,ABBCS}. In these works, an explicit description of $\mi_\alpha$ was provided when $D$ is a bidisc, or even a polydisc in some cases. This description was then applied to characterize the $\phi$ and $\alpha$ for which $H^2(\mi_\alpha)= L^2(\mi_\alpha)$ when $D$ is the bidisc. 
In this paper we provide an extension of these results. 
In fact, in Theorem~\ref{teo:3} we provide an explicit description of $\mi_\alpha$ when $\phi$ is a rational inner function and $D$ is arbitrary. This description is the natural analogue of the known description of Clark measures associated with finite Blaschke products, that is, rational inner functions on $\Db$ (even though this fact is less apparent in~\cite{BCS,ABBCS}). 
We then characterize the $\phi$ and $\alpha$ for which $H^2(\mi_\alpha)= L^2(\mi_\alpha)$ when $D$ is a polydisc. For more general domains the situation is more intricate and we only provide some sufficient conditions in Proposition~\ref{prop:42} and some necessary conditions in Proposition~\ref{prop:40}. Even though these conditions do not provide a complete characterization, we are still able to show that the equality $H^2(\mi_\alpha)= L^2(\mi_\alpha)$ may fail for rational inner functions even when $D$ is irreducible (cf.~Example~\ref{ex:1}), contrary to what one may have hoped, since counterexamples were so far available only on reducible domains.

Here is a plan of the paper. 
In Section~\ref{sec:2}, we provide a description of $\mi_\alpha$ for general rational inner functions and a characterization of the equality  $H^2(\mi_\alpha)= L^2(\mi_\alpha)$ when $D$ is a polydisc. We have chosen to separate these results from the remaining ones since they may be stated and proved keeping the preliminaries to a minimum and may therefore be more accessible to the readers who are only interested in the case of polydiscs and are less acquainted with the theory of general symmetric domains.

In Section~\ref{sec:3}, we briefly recall some basic notions on Jordan triple systems which are necessary to deal with general symmetric domains.
In Section~\ref{sec:4}, we prove our sufficient and necessary conditions for the equality  $H^2(\mi_\alpha)= L^2(\mi_\alpha)$ for general symmetric domains, and  provide an example in which this equality fails even if the domain is irreducible.

\section{Rational Inner Functions on Polydiscs}\label{sec:2}

\begin{nott}\label{not:1}
	Throughout the paper, we shall use the following notation.
	We shall denote with $\Db$ the unit disc in $\C$ and by $\T$ its boundary.
	
	Given a polynomial $p$ (or a linear function) on $\C^n$, we shall define $\overline p\colon z \mapsto \overline{p(z)}$ and $p^\sharp \colon z \mapsto \overline{p(\overline z)}$.\footnote{Notice that, with this convention, the $\C$-linear  conjugate of a $\C$-linear function $L$ on $\C^n$ is    $L^\sharp$ and not $\overline L$.}
	
	We  normalize the Hausdorff measure $\Hc^k$ on $\C^n$ so that the measure of $k$-dimensional unit cubes is $1$ for $k=0,\dots, 2n$.
	
	We shall denote with $D$ a (convex) circular  bounded symmetric domain in $\C^n$, that is, an open convex subset of $\C^n$ such that $\T D=D$ and such that every $z\in D$ is the unique fixed point of an involutive biholomorphism of $D$. We shall define $r$ so that $\sup_{z\in D}\abs{z}^2=r$.\footnote{This choice is due to the fact that, in Section~\ref{sec:3} and~\ref{sec:4}, $r$ will be the rank of $D$.}
	
	We shall denote with $K$ the component of the identity of the group of linear automorphisms of $D$; we shall also assume that $K\subseteq U(n)$ (cf.~\cite[Corollary 3.16 and Theorem 4.1]{Loos}). 
	
	We shall denote with $\bD$ the set of extremal points of $D$, so that $\bD$ equals $\partial D\cap \partial B(0, \sqrt r)$ and is also the \v Silov boundary of $D$ (cf.~\cite[Theorem 6.5]{Loos}), that is, the smallest closed subset of $\partial D$ such that $\sup_D\abs{f}= \max_{\bD} \abs{f}$ for every continuous function $f$ on the closure $\Cl(D)$ of $D$ which is holomorphic on $D$. Observe that $\bD$ is a real algebraic set and a compact connected analytic submanifold of $\C^n$ on which $K$ acts transitively  (cf.~\cite[Theorem 6.5]{Loos}). We shall generally use $z$ to denote an element of $D$ and $\zeta$ to denote an element of $\bD$.
	
	We shall denote with $\beta$ the normalized $K$-invariant measure on $\bD$, so that $\beta=c\chi_{\bD}\cdot \Hc^m$, where $m$ is the dimension of $\bD$ and $c=1/\Hc^m(\bD)$.\footnote{If $\mi$ is a Radon measure and $f\in L^1_\loc(\mi)$, we denote with $f\cdot \mi$ the Radon measure such that $\mi(A)=\int_A f\,\dd\mi$ for every Borel set $A$.} 
	
	We shall denote by $\phi\colon D\to \Db$ a rational inner function, that is, a rational function such that $\phi_1$, the pointwise limit of the $\phi_\rho=\phi(\rho\,\cdot\,)$ for $\rho \to 1^-$ on $\bD$ (where it exists), takes unimodular values $\beta$-almost everywhere on $\bD$.
	
	We shall denote with $\mi_\alpha$ the Clark measure associated with $\phi$ and $\alpha\in \T$, that is, the unique positive Radon measure on $\bD$ such that
	\[
	(\Pc \mi_\alpha)(z)=\Re \frac{\alpha+\phi(z)}{\alpha-\phi(z)}= \frac{1-\abs{\phi(z)}^2}{\abs{\alpha-\phi(z)}^2}
	\]
	for every $z\in D$, where $\Pc \mi_\alpha$ denotes the Poisson--Szeg\H o integral of $\mi_\alpha$.
	
	Finally, we  denote with $H^2(\mi_\alpha)$ the closure in $L^2(\mi_\alpha)$ of the space of holomorphic polynomials on $\C^n$.
\end{nott}

Notice that, when $D$ is the polydisc $\Db^n$, then $r=m=n$, $K=\T^n$ acting by componentwise multiplication, and $\bD=\T^n$, while $\Pc$ is simply the operator whose kernel is the direct product of the classical Poisson kernels on each component (cf.~\eqref{eq:3}).  
We prefer to avoid recalling the definition of $\Pc$ in the general case since we shall not use any of its properties. In fact, for the purposes of this paper, we only need the information summarized in the following proposition (cf.~\cite[Theorem 3.1, Corollary 3.3,  Remark 4.2, Proposition 4.6 and 4.7]{BoundedClark}).
Cf.~\cite{BoundedClark} for more information on Clark measures on bounded symmetric domains; notice, though, that~\cite{BoundedClark} adds the assumption $r=1$ to simplify some expressions. Since, however, most objects are defined in a normalized way, the   dilations needed to consider the case $r\neq 1$ are essentially harmless.

\begin{prop}\label{prop:0}
	Take $\alpha\in \T$. Then, $\mi_\alpha$ is absolutely continuous with respect to $\Hc^{m-1}$. In addition, if for every $\xi\in \widehat \bD=\bD/\T$ (where $\T$ acts on $\bD$ by multiplication) we denote with $\mi_{\alpha,\xi}$ the unique positive Radon measure on $\pi^{-1}(\xi)$ such that
	\[
	\int_{\pi^{-1}(\xi)} \frac{r^2-\abs{z}^2}{\abs{\zeta-z}^2}\,\dd \mi_{\alpha,\xi}(\zeta)= \frac{1-\abs{\phi(z)}^2}{\abs{\alpha-\phi(z)}^2}
	\]
	for every $z\in \Db\pi^{-1}(\xi)$,\footnote{That is, $\mi_{\alpha,\xi}$ is the classical Clark measure associated with the restriction of $\phi$ to the disc $ \Db\pi^{-1}(\xi)$ and $\alpha$.}	where $\pi\colon \bD\to \widehat \bD$ denotes the canonical projection, then the mapping $\xi\mapsto \mi_{\alpha,\xi}$ is vaguely continuous and 
	\[
	\mi_\alpha= \int_{\widehat \bD} \mi_{\alpha,\xi}\,\dd \widehat \beta(\xi)
	\]
	where $\widehat\beta$ denotes the image measure  of $\beta$ under $\pi$.
	Finally, $\mi_\alpha$ is concentrated on $\phi_1^{-1}(\alpha)$.
\end{prop}

We may now provide a description of the Clark measures associated with rational inner functions on $D$.

\begin{teo}\label{teo:3}
	Let $p,q$ be two coprime polynomials such that $\phi=p/q$ on $D$ and take $\alpha\in \T$. Then, the following hold:
	\begin{enumerate}
		\item[\textnormal{(1)}] the algebraic set $N\coloneqq \Set{\zeta\in \bD\colon q(\zeta)=0}$   has (real) dimension $\meg m-2$;
		
		\item[\textnormal{(2)}] the closure $\Cl(\phi_1^{-1}(\alpha))$ of $\phi_1^{-1}(\alpha)$ is $V_\alpha\coloneqq\Set{\zeta\in \bD\colon p(\zeta)=\alpha q(\zeta)}$, and has (real) dimension $m-1$;
		
		\item[\textnormal{(3)}] $\phi_1$ is defined on the whole of $\bD$ and attains unimodular values thereon;
		
		\item[\textnormal{(4)}] setting $c\coloneqq \Hc^m(\bD)$,
		\begin{equation}\label{eq:1}
		\mi_\alpha= \frac{2\pi}{c\abs{\nabla \phi}} \chi_{V_\alpha\setminus N} \cdot \Hc^{m-1}= \frac{2\pi}{c\abs{\nabla \phi}} \chi_{\phi^{-1}_1(\alpha)} \cdot \Hc^{m-1}= \frac{2\pi}{c\abs{\nabla \phi}} \chi_{V_\alpha} \cdot \Hc^{m-1}.
		\end{equation}
	\end{enumerate}
\end{teo}

Observe  that if $\zeta\in \bD \setminus N$, then $\nabla \phi(\zeta)$ is defined (and equal to $\nabla p(\zeta)/q(\zeta)-p(\zeta)\nabla q(\zeta)/q(\zeta)^2$). Hence, the expressions for $\mi_\alpha$ provided in~\eqref{eq:1} make sense.  

This result extends~\cite[Proposition 3.9]{BCS} and~\cite[Theorems 3.3, 3.5, and 3.8]{ABBCS} to the case of general rational inner functions and general bounded symmetric domains. The proof of~(4) is different (and somewhat simpler), even though the statement is a little less precise, since $\phi^{-1}(\alpha)$ is not described as a finite union of analytic curves (in the case of the bidisc) or analytic manifolds (in the case of the polydisc, when possible). Nonetheless, such representations essentially follow (even though in a more abstract way) from the fact that $\phi^{-1}_1(\alpha)$ differs from the algebraic set $V_\alpha$ by a subset of $N$, which is an algebraic set of   dimension $\meg m-2$.

Before we pass to the proof, we need a few lemmas.

\begin{lem}\label{lem:14}
	Let $V$ be a complex algebraic set in $\C^n$.  Then $m-\dim_\R (V\cap \bD)\Meg n-\dim_\C V$.
\end{lem}

\begin{proof}
	Observe that $V$ is the union of a complex manifold which is locally closed in the Zariski topology and of dimension $\dim_\C V$ and of a complex algebraic set of strictly lower dimension (the singular set of $V$). 
	Iterating this decomposition, we see that $V$ is a finite union of pairwise disjoint  complex manifolds $V_j$ which are locally closed in the Zariski topology and have dimension $\meg \dim_\C V$. 
	Now, each $V_j\cap \bD$ is a (real) semi-algebraic set, hence may be decomposed as a finite union of pairwise disjoint (real) analytic manifolds $M_{j,k}$ of dimension $\meg \dim_\R( V_j\cap \bD)$ (cf., e.g.,~\cite[Theorem 1.3]{Coste}). 
	Now, if $\zeta\in M_{j,k}$, then $T_\zeta M_{j,k}\subseteq T_\zeta V_j \cap T_\zeta\bD$, where $T_\zeta$ is used to denote tangent spaces at $\zeta$. 
	Since $T_\zeta \bD$ generates $\C^n$ over $\C$ (as is most easily seen in the Siegel domain realization of $D$, cf.~\cite[Theorem 10.8]{Loos}) and since and $T_\zeta V_j$ is a complex vector subspace of $\C^n$, we see that $m-\dim_\R( T_\zeta V_j \cap T_\zeta\bD) \Meg n-\dim_\C T_\zeta V_j$,\footnote{We may find a basis $v_1,\dots, v_q$ of $T_\zeta V_j \cap T_\zeta\bD\cap (i T_\zeta \bD)$ over $\C$, and complete it to a basis $v_1,\dots, v_q,i v_1,\dots, i v_q, v'_1,\dots, v'_p$ of $T_\zeta V_j \cap T_\zeta\bD$ over $\R$. We may also complete $v_1,\dots, v_q,  v'_1,\dots, v'_p$ to a basis $v_1,\dots, v_q, v'_1,\dots, v'_p,v''_1,\dots, v''_s$ of $\C^n$ over $\C$, with $v''_1,\dots v''_s\in T_\zeta \bD$, so that $m\Meg 2 q+p+s$, $\dim_\R M_{j,k}=2 q+p$, $n=q+p+s$, $\dim_{\C} V_j\Meg q+p$, whence $m-\dim_{\R}M_{j,k}\Meg s\Meg n-\dim_\C V_j$
	} whence $m-\dim_\R M_{j,k}\Meg n-\dim_\C V_j$. The assertion follows, since $\dim_\R (V\cap \bD)=\max_{j,k} \dim_\R M_{j,k}$.
\end{proof}

\begin{lem}\label{lem:15}
	Assume that $D$ is of tube type. Then, \textnormal{(1)} and \textnormal{(2)} of Theorem~\ref{teo:3} hold.
\end{lem}

\begin{proof}
	Observe that $n=m$, since $D$ is of tube type. In addition, since $\phi$ is bounded on $D$, if $\zeta\in \bD$ and $q(\zeta)=0$, then $p(\zeta)=0$, so that $N$ is also the algebraic set defined by $p$ and $q$ on $\bD$. 
	Let $V$ be the (complex) algebraic subset of $\C^n$ defined by $p$ and $q$, and observe that $V$ has (complex) codimension $\Meg 2$ in $\C^n$ since $p$ and $q$ are coprime. 
	Consequently, Lemma~\ref{lem:14} shows that also the (real) algebraic set $N= \bD \cap V$ has (real) codimension $\Meg 2$ in $\bD$. 
	This proves (1). In a similar way, one proves that $V_\alpha$ has dimension $\meg n-1$. Since   $\mi_\alpha$ is non-zero, concentrated on $\phi_1^{-1}(\alpha)$, and absolutely continuous with respect to $\Hc^{n-1}$ by   Proposition~\ref{prop:0}, it is then clear that $\Hc^{n-1}(\phi_1^{-1}(\alpha))>0$. 
	Observe also that $\phi_1^{-1}(\alpha) \subseteq V_\alpha=\phi_1^{-1}(\alpha)\cup N$: in fact, if $\zeta\in \bD$ and $q(\zeta)\neq 0$, then $\phi_1(\zeta)=p(\zeta)/q(\zeta)$, so that $\zeta\in \phi_1^{-1}(\alpha)$ if and only if $\zeta\in V_\alpha$, and clearly $N\subseteq V_\alpha$. 
	Then, $V_\alpha$ has dimension $n-1$. 
	
	In order to show that $V_\alpha=\Cl(\phi_1^{-1}(\alpha))$, we proceed by contradiction.\footnote{The reader only interested in the case of polydiscs may skip the rest of the proof, since this result was proved in~\cite[Theorem 2.4]{ABBCS} in this case.} Take $\zeta\in V_\alpha \setminus \Cl(\phi_1^{-1}(\alpha))=N\setminus \Cl(\phi_1^{-1}(\alpha))$, and let $\Cc_{-\zeta}$ be the Cayley transform which maps $D$ onto a tube domain $T_\Omega=\R^n+i \Omega$ for some open convex cone $\Omega\subseteq \R^n$, maps $\bD\cap \mathrm{dom} \,\Cc_{-\zeta}$ onto $\R^n$, and $\zeta$ to $0$. With the notation of Section~\ref{sec:3}, and assuming (up to a rotation) that $\overline z=\{-\zeta,z,-\zeta\}$ for every $z\in \C^n$, one has $\Cc_{-\zeta}(z)= -i(z+\zeta)^{-1}(z-\zeta)$, where $\C^n$ is endowed with the Jordan algebra product $(z,z')\mapsto \{z,-\zeta,z'\}$, and $\Omega=\Set{x^2\colon x\in \R^n, \det(x)\neq 0}$ (cf.~\cite[Theorem 10.8]{Loos}).
	In addition, let $\Cc_\alpha$ be the Cayley transform which maps $\Db$ onto $\C_+$ and $\T\setminus \Set{\alpha}$ onto $\R$ (namely, $\Cc_\alpha(w)=i \frac{\alpha+w}{\alpha-w}$). 
	Then, $\psi\coloneqq\Cc_\alpha\circ \phi\circ \Cc^{-1}_{-\zeta}$ is a rational function on $T_\Omega$ and takes real values on $\R^n\setminus \Cc_{-\zeta}(\phi_1^{-1}(\alpha)\cup N)$. Therefore, $\psi$ is also holomorphic in $-T_\Omega$ by reflection. 
	Now, let $(e_1,\dots, e_n)$ be a basis of $\R^n$ with elements in $\Omega$, and let $L$ be the $\C$-linear automorphism of $\C^n$  which maps $e_j$ to the $j$-th element of the canonical basis of $\C^n$. 
	Then, $L^{-1}(\C_+^n)\subseteq T_\Omega$. We may then apply the arguments in the proof of~\cite[Theorem 2.6]{BPS} to $\psi\circ L^{-1}$, reaching a contradiction. Thus, (2) follows.
\end{proof}

\begin{proof}[Proof of Theorem~\ref{teo:3}.]
	Observe that, if $D'$ and $D''$ denote the products of the irreducible factors of $D$ which are and are not of tube type, respectively, and if $\pr\colon D\to D'$ denotes the  linear projection  with kernel $D''$, then $\phi=\psi\circ \pr$ for some rational inner function $\psi$ on $ D'$, thanks to~\cite[Theorem 3.3]{KoranyiVagi}. Therefore, (1) and (2) follow from Lemma~\ref{lem:14}.
	
	Now, take $\xi\in \widehat \bD$ and observe that the restriction of $\phi$ to $\Db_\xi$ is a rational function. 
	In addition, with the notation of Proposition~\ref{prop:0}, there is a $\widehat \beta$-negligible subset $N'$ of $\widehat \bD$ such that the restriction of $\phi$ to $\Db_\xi\coloneqq \Db \pi^{-1}(\xi)$ is inner as well, hence a finite Blaschke product, for every $\xi\in \widehat \bD\setminus N'$. Consequently, by~\cite[Example 9.2.4 (2)]{CimaMathesonRoss},
	\[
	\mi_{\alpha,\xi}= \sum_{ \zeta\in \phi^{-1}_1(\alpha)\cap \pi^{-1}(\xi) } \frac{1}{\abs{(\partial_\zeta \phi)_1(\zeta)}} \delta_\zeta
	\]
	for every $\xi\in \widehat \bD\setminus N'$, where $(\partial_\zeta \phi)_1(\zeta)$ denotes the pointwise limit of $(\partial_\zeta \phi)(\rho\zeta)$ for $\rho \to 1^-$. 
	Observe that both $\phi_1(\zeta)$ and $(\partial_\zeta \phi)_1(\zeta)$ exist for every $\zeta\in \pi^{-1}(\xi)$, since they are  the radial limits at $1$ of the finite Blaschke product $\Db\ni w \mapsto \phi(w \zeta)\in \Db$ and of its derivative, respectively. 
	In addition, observe that 
	\[
	\card( \phi^{-1}_1(\alpha)\cap \pi^{-1}(\xi))\meg \card( V_\alpha\cap \pi^{-1}(\xi) )\meg \max(\deg p, \deg q).
	\]
	Since the mapping $\xi \mapsto \mi_{\alpha,\xi}$ is vaguely continuous on $\widehat \bD$ by Proposition~\ref{prop:0}, it then follows that $\mi_{\alpha,\xi}$ consists of the sum of at most $ \max(\deg p, \deg q) $ point masses for \emph{every} $\xi\in \widehat \bD$. In other words, the restriction of $\phi$ to $\Db_\xi$ is a rational \emph{inner} function for every $\xi\in \widehat \bD$, so that we may take $N'=\emptyset$.  The previous remarks then show that $\phi_1$ is defined on the whole of $\bD$ and has unimodular values thereon, that is, (3) holds. In addition, $(\partial_\zeta \phi)(\zeta)\neq 0$ for every $\zeta\in V_\alpha \setminus N$.
	Therefore,
	\begin{equation}\label{eq:6}
		\mi_\alpha= \int_{\widehat \bD} \mi_{\alpha,\xi}\,\dd \widehat \beta(\xi)= \int_{\widehat \bD} \sum_{ \zeta\in \phi^{-1}_1(\alpha)\cap \pi^{-1}(\xi) } \frac{1}{\abs{(\partial_\zeta \phi)_1(\zeta)}} \delta_\zeta\,\dd \widehat \beta(\xi)
	\end{equation}
	by Proposition~\ref{prop:0}. 
	Now, endow the (real) analytic manifold $\bD$ with the Riemannian metric induced by the scalar product of $\C^n$; by the area formula, the corresponding Riemannian volume is $\chi_{\bD}\cdot \Hc^m=c\beta$ (recall that $c=\Hc^m(\bD)$).  
	Since $\T$ acts by isometries on $\bD$, we may endow $\widehat \bD$ with the quotient metric induced by the submersion $\pi\colon \bD\to \widehat \bD$, so that the corresponding Riemannian volume is  $c'\widehat \beta$ for some $c'>0$ by $K$-invariance.\footnote{Observe that the action of $\T$ commutes with that of $K$, so that $K$ acts transitively on $\widehat \bD$ as well.}. Then,  the disintegration (cf.~\cite[Remark 2.5]{BoundedClark})
	\[
	\beta=\int_{\widehat \bD} \beta_\xi \,\dd \widehat \beta(\xi),
	\]
	where $\beta_\xi=\frac{1}{2\pi\sqrt r}\chi_{\pi^{-1}(\xi)}\cdot \Hc^1$, compared with the smooth coarea formula associated with $\pi\colon \bD\to \widehat \bD$, shows that $c=2 \pi \sqrt r c'$.
	If we now apply the coarea formula to $ \phi^{-1}_1(\alpha)$,\footnote{Notice that $\phi^{-1}_1(\alpha)$ is $\Hc^{m-1}$-measurable and rectifiable, since $\phi^{-1}_1(\alpha)\cup N$ is the algebraic set $\Set{ \zeta\in \bD\colon p(\zeta)=\alpha q(\zeta)}$ and $\Hc^{m-1}(N)=0$.} we obtain
	\begin{equation}\label{eq:7}
		f \chi_{\phi_1^{-1}(\alpha)}\cdot \Hc^{m-1}= c'\int_{\widehat \bD} \sum_{ \zeta\in \phi^{-1}_1(\alpha)\cap \pi^{-1}(\xi) }  \delta_\zeta\,\dd \widehat \beta(\xi),
	\end{equation}
	where $f(\zeta)$ is the coarea factor at $\zeta$ of the restriction of $\pi$ to $\phi_1^{-1}(\alpha)$. Observe that, if $\zeta\in \phi_1^{-1}(\alpha)\setminus N$, then $\phi_1^{-1}(\alpha)$ is a (real) analytic manifold in a neighbourhood of $\zeta$ by the implicit function theorem, since $\nabla \phi(\zeta)$ is defined and non-zero -- in fact, $(\partial_\zeta \phi)(\zeta)\neq 0$ by the previous remarks. 
	In particular,  $T_\zeta \phi^{-1}_1(\alpha)=\nabla \phi(\zeta)^\perp\cap T_\zeta \bD$, where $T_\zeta$ is used to denote tangent spaces at $\zeta$. 
	Using the fact that the  differential of $\pi$ on $T_\zeta \bD$ is a surjective partial isometry whose kernel is the line generated by $\zeta$ (and that $\abs{\zeta}=\sqrt r$), we then see that $f(\zeta)= \frac{\abs{\langle \nabla\phi(\zeta)\vert \zeta \rangle}}{\abs{\zeta}\abs{\nabla \phi(\zeta)}} = \frac{\abs{\partial_\zeta \phi(\zeta)}}{\sqrt r\abs{\nabla \phi(\zeta)}}$. Comparing~\eqref{eq:6} and~\eqref{eq:7}, we find
	\[
	\mi_\alpha= \frac{2\pi}{c\abs{\nabla \phi}} \chi_{\phi^{-1}_1(\alpha)} \cdot \Hc^{m-1},
	\]
	whence (4).
\end{proof}

We now pass to the problem of characterizing when $H^2(\mi_\alpha)=L^2(\mi_\alpha)$, when $D$ is a polydisc.

\begin{teo}\label{teo:4}
	Assume that $D=\Db^n$ and take $\alpha\in \T$. Then, $H^2(\mi_\alpha)=L^2(\mi_\alpha)$ if and only if   $\Cl(\phi_1^{-1}(\alpha))$ does not contain any set of the form $\pr_j^{-1}(V)$, where $j\in \Set{1,\dots, n}$, $V$ is a (real) algebraic hypersurface in $\T^{n-1}$, and $\pr_j\colon \T^n\to \T^{n-1}$ is the   projection $(\zeta_1,\dots, \zeta_n)\mapsto (\zeta_1,\dots, \zeta_{j-1}, \zeta_{j+1},\dots, \zeta_n)$.
\end{teo}

This extends~\cite[Theorems 4.2 and 4.3]{ABBCS} to the case of general rational inner functions on general polydiscs. The proof is similar but more involved. We begin with a few lemmas.

\begin{lem}\label{lem:12}
	Let $\phi\colon D\to \Db$ be a holomorphic function, and take $\alpha\in \T$. If $f,g\in H^2(\mi_\alpha)$ and $f$ is bounded, then $f g\in H^2(\mi_\alpha)$.
\end{lem}

\begin{proof}
	Take two sequences $(p_j)$, $(q_j)$ of holomorphic polynomials on $\C^n$ such that $p_j\to f$ and $q_j\to g$ in $L^2(\mi_\alpha)$. Then, $p_j q_k\to f q_k$ in $L^2(\mi_\alpha)$, so that $f q_k\in H^2(\mi_\alpha)$ for every $k\in\N$. Since $f$ is bounded, it is also clear that $ f q_k\to  f g$ in $L^2(\mi_\alpha)$, whence the conclusion.
\end{proof}

\begin{lem}\label{lem:13}
	Take $\eps\in (0,1)$ and let $C$ be the convex envelope of $B_\C(0,\eps)\cup \Set{1}$. If $p$ is a holomorphic polynomial in $\C$ of degree $k\in \N$ such that $p(z)\neq 0$ for every $z\in C$, then $\abs{p(x)}\Meg (\eps/2)^k \abs{p(1)}$ for every $x\in [0,1]$.
\end{lem}

\begin{proof}
	\textsc{Step I} Let us first prove that, if $A\coloneqq \C\setminus C$ and $d\colon A\to (0,+\infty)$ is defined so that $d(z)=\min_{x\in [0,1]} \abs{z-x}$ for every $z\in A$, then  $d(z)\Meg (\eps/2) \abs{z-1}$. In other words, let us prove the assertion for  $k=1$. 
	To this aim, fix $z\in A$ and observe that $[0,1]\ni x \mapsto\abs{z-x}$  takes its minimum at $0$ if $\Re z \meg 0$, at $1$ if $\Re z \Meg 1$, and at $\Re z$ otherwise. 
	In particular, $d(z)= \abs{z-1}$ if $\Re z \Meg 1$, while 
	\[
	d(z)=\abs{\Im z} \Meg \frac{\abs{\Im z}}{2}+\frac{\eps}{2}(1-\Re z)\Meg \frac{\eps}{2}\abs{z-1}
	\]
	if $\Re z\in (0,1)$, since $C$ contains the triangle with vertices $\pm \eps i$ and $1$.
	Finally, 
	\[
	d(z)= \abs{z}\Meg \frac{\abs{z}}{1+\abs{z}}\abs{z-1} \Meg \frac{\eps}{1+\eps}\abs{z-1}\Meg \frac{\eps}{2}\abs{z-1}
	\]
	if $\Re z \meg 0$,	since $\abs{z}\Meg \eps$ (and $\eps\in (0,1)$).   
	
	\textsc{Step II} Take $p$ as in the statement. Take $\alpha\in \C \setminus \Set{0}$ and $z_1,\dots, z_k\in A$  such that $p(z)= \alpha \prod_{j=1}^k (z-z_j)$. Then,~\textsc{step I} shows that 
	\[
	\abs{p(x)}=\abs{\alpha} \prod_{j=1}^k \abs{x-z_j} \Meg \abs{\alpha} \prod_{j=1}^k \left(\frac{\eps}{2}\right)^k  \abs{1-z_j}  \Meg  \left(\frac{\eps}{2}\right)^k   \abs{p(1)}
	\]
	for every $x\in [0,1]$,	so that the assertion follows.
\end{proof}

\begin{proof}[Proof of Theorem~\ref{teo:4}.]
	\textsc{Step I} Assume first that $V_\alpha\coloneqq\Cl(\phi_1^{-1}(\alpha))$ does not contain any set of the form $\pr_j^{-1}(V)$, where $j=1,\dots, n$ and  $V$ is a (real) algebraic hypersurface in $\T^{n-1}$.  
	Then,  let us prove that for every $j=1,\dots, n$ there is a rational function $r_j$ on $\C^n$, defined (and holomorphic) on $D$, such that $r_j(\zeta)$ is defined and equals $\overline{\zeta_j}$   for $\Hc^{n-1}$-almost every $\zeta\in \phi^{-1}_1(\alpha)$ (as we shall see, this implies that $r_j\in H^2(\mi_\alpha)$ for every $j=1,\dots, n$). 
	Take two coprime holomorphic polynomials $p,q$ on $\C^n$ such that $\phi=p/q$ on $D$, so that $ V_\alpha$ is the zero set of $p-\alpha q$ on $\bD=\T^n$, thanks to Theorem~\ref{teo:3}.
	For simplicity, we may assume that $j=1$, and write $z=(z_1,z')$ for every $z\in D$. Then, observe that we may take (unique) holomorphic polynomials $p_{1}, p_{2}, q_{1}, q_{2}$ such that  $p(z)=p_{1}(z') + z_1 p_{2}(z) $ and $q(z)=q_1(z')+z_1 q_2(z)$, so that
	\[
	\phi(z)= \frac{p_{1}(z') + z_1 p_{2}(z) }{q_1(z')+z_1 q_2(z)}
	\]
	for every $z=(z_1,z')\in D$.
	
	Let us  prove that $p_1-\alpha q_1$ vanishes \emph{nowhere} on $\Db^{n-1}$ and that the zero set of $  (p_1 -\alpha q_1)\circ \pr_1$ in $V_\alpha$ has dimension $\meg n-2$. 
	Observe first that, if $z'\in \Db^{n-1}$ and $p_1(z')=\alpha q_1(z')$, then $(0,z')\in D$ and $\phi(0,z')=\frac{p_1(z')}{q_1(z')}=\alpha\in \T$, which is absurd since $\phi$ maps $D$ in $\Db$. This proves that $p_1-\alpha q_1$ vanishes nowhere on $\Db^{n-1}$. 
	Next, assume by contradiction that the zero set of $ (p_1 -\alpha q_1)\circ \pr_1$ in $V_\alpha$ has dimension $> n-2$, hence $n-1$ (since $V_\alpha$ has dimension $n-1$). 
	Then, the joint zero set of $ (p_1 -\alpha q_1)\circ \pr_1$ and $p-\alpha q$ in $\bD$ has real dimension $n-1$, so that Lemma~\ref{lem:14} implies that the joint zero set of $ (p_1 -\alpha q_1)\circ \pr_1$ and $p-\alpha q$  in $\C^n$ has complex dimension $n-1$. 
	Then, there must be a non-constant holomorphic polynomial $p_3$ which divides both $ (p_1 -\alpha q_1)\circ \pr_1$ and $p-\alpha q$, so that  $ p_3=p_4\circ \pr_1$ for some holomorphic polynomial $p_4 $ on $\C^{n-1}$. 
	Then, the zero set of $p_3$ in $\bD$ is   of the form $\T\times V$, where $V$ is the zero set of $p_4$ in $\T^{n-1}$, and is contained in $V_\alpha$. If we further assume that $ [(p_1 -\alpha q_1)\circ \pr_1]/p_3$ and $(p-\alpha q)/p_3$ are coprime polynomials (as we may), by means of Lemma~\ref{lem:14} it then follows that $V$ must have dimension $n-2$: contradiction.
	
	We may then define
	\[
	r_1(z)\coloneqq \frac{\alpha q_2 (z)- p_2(z) }{ p_1(z')- \alpha q_1(z')}
	\]
	for $z=(z_1,z')\in D$,	so that $r_1$ is defined and holomorphic on $D$ and $r_1(\zeta)=\overline{\zeta_1}$	for $\Hc^{n-1}$-almost every $\zeta=(\zeta_1,\zeta')\in V_\alpha$.
	
	For every $\rho \in (0,1)$ we may also define the rational function
	\[
	r_{1,\rho}(z)\coloneqq \frac{\alpha q_2 (z)- p_2(z) }{ p_1(\rho z')- \alpha q_1(\rho z')},
	\]
	so that  $r_{1,\rho}$ is defined and holomorphic in a neighbourhood of the closure $\Cl(D)$ of $D$, and   $r_{1,\rho}(\zeta)\to r_1(\zeta)=\overline{\zeta_1}$ for $\Hc^{n-1}$-almost every $\zeta=(\zeta_1,\zeta')\in V_\alpha$. 
	In addition, for $\Hc^{n-1}$-almost every $\zeta=(\zeta_1,\zeta')\in V_\alpha$ one has $p_1(\zeta')- \alpha q_1(\zeta')\neq 0$, so that $\abs{r_{1,\rho}(\zeta)}\meg 2^{-k}\abs{r_1(\zeta)}=2^{-k}$ (where $k=\deg (p_1-\alpha q_1)$), thanks to Lemma~\ref{lem:13}. 
	Consequently, by the dominated convergence theorem we see that $r_{1,\rho}\to r_1$ in $L^2(\mi_\alpha)$ for $\rho \to 1^-$. Since $r_{1,\rho}\in H^2(\mi_\alpha)$ for every $\rho\in (0,1)$ (cf., e.g.,~\cite[Remark 2.17]{BoundedClark}), this proves that $r_1\in H^2(\mi_\alpha)$.
	
	Now, Lemma~\ref{lem:12} and the previous remarks show that for every $\gamma,\delta\in \N^n$ the monomial $\zeta \mapsto \zeta^\gamma \overline{\zeta}^\delta$ belongs to $H^2(\mi_\alpha)$. Consequently, $C(\bD)\subseteq H^2(\mi_\alpha)$ by the Stone--Weierstrass theorem, so that $L^2(\mi_\alpha)=H^2(\mi_\alpha)$.

	\textsc{Step II} Assume that  $\Cl(\phi^{-1}_1(\alpha))\supseteq\T\times V$, where $V$ is an algebraic hypersurface in $\T^{n-1}$ (for the sake of simplicity, we are taking $j=1$; the other cases are similar). 
	Assume by contradiction that $L^2(\mi_\alpha)=H^2(\mi_\alpha)$. 
	Then, there is a sequence $(p_k)$ of holomorphic polynomials on $\T^n$ such that $p_k\to \overline L$ in $L^2(\mi_\alpha)$, where $L(z)=z_1$ for every $z\in \C^n$. 
	By Theorem~\ref{teo:3}, up to a subsequence we may assume that $p_k(\,\cdot\,,\zeta')\to \overline{\,\cdot\,}$ in $L^2(\T, \abs{\nabla \phi(\,\cdot\,,\zeta')}^{-1}\cdot \Hc^1)$ for $\Hc^{n-2}$-almost every $\zeta'\in V$. 
	In particular,  if $\phi=p/q$ for some coprime holomorphic polynomials $p$ and $q$ on $\C^n$, we may take $\zeta'\in V$ so that $q(\,\cdot\,,\zeta')$ has only finitely many zeros on $\T$ (cf.~Theorem~\ref{teo:3}) and   $p_k(\,\cdot\,,\zeta')\to \overline{\,\cdot\,}$ in $L^2(\T, \abs{\nabla \phi(\,\cdot\,,\zeta')}^{-1}\cdot \Hc^1)$. 
	Then, $\abs{\nabla \phi(\,\cdot\,,\zeta')}^{1/2}$ belongs to $L^s(\T)$ for some $s\in (0,1)$, so that    $p_k(\,\cdot\,,\zeta')\to \overline{\,\cdot\,}$ in $L^{s_2}(\T)$, where $\frac{1}{s_2}=\frac{1}{2}+\frac{1}{s}$.  
	Since the functions $p_k(\,\cdot\,,\zeta')$ belong to the boundary value space of $H^{s_2}(\Db)$, which is closed in $L^{s_2}(\T)$, this leads to a contradiction.
\end{proof}

\section{Jordan Triple Systems}\label{sec:3}

In order to provide some (partial) analogues of Theorem~\ref{teo:4} for general symmetric domains, we need to use the Jordan triple system on $\C^n$ induced by $D$.  We refer the reader to~\cite{Loos} for a more comprehensive discussion of this subject. See also~\cite{MackeyMellon, SiegelClark}.  

\begin{deff}
	A (finite-dimensional) positive Hermitian Jordan triple system is a (finite-dimensional) complex vector space $Z$ endowed with a `triple product' $\{\,\cdot\,,\,\cdot\,,\,\cdot\,\}\colon Z\times Z\times Z\to Z$ such that the following hold:
	\begin{itemize}
		\item $\{x,y,z\}$ is $\C$-bilinear and symmetric in $(x,z)$, and $\C$-antilinear in $y$;
		
		\item $[D(a,b), D(x,y)]=D(D(a,b) x,y)-D(x,D(b,a)y) $ for every $a,b,x,y,z\in Z$, where $D(a,b)=\{a,b,\,\cdot\,\}$ for every $a,b\in Z$;
		
		\item if $\{x,x,x\}=\lambda x$ for some non-zero $x\in Z$ and some $\lambda\in \C$, then $\lambda>0$.
	\end{itemize}
	We also define  $Q(x)\coloneqq \{x,\,\cdot\,,x\}$ for every $x\in Z$. The Jordan triple system $Z$ is said to be simple if it is not isomorphic to the product of two non-trivial (positive Hermitian) Jordan triple systems.
	
	Denote by $\Aut(Z)$ the group of automorphisms of $Z$, that is, the group of the $\C$-linear automorphisms $u$ of the vector space $Z$ such that $\{u x,u y, uz\}=u\{x,y,z\}$ for every $x,y,z\in Z$, and let $\Aut_0(Z)$ be the component of the identity in $\Aut(Z)$. 
	
	An $e\in Z$ is called a tripotent if $\{e,e,e\}=e$.\footnote{There is no general agreement on the definition of a tripotent; for instance, $e$ is a tripotent according to~\cite{Loos} if and only if $\sqrt 2 e $ is a tripotent according to our definition. In general, there are also other powers of $\sqrt 2$ that may differ from one reference to another.} Two tripotents $e,e'\in Z$ are said to be orthogonal if $D(e,e')=0$ (equivalently, if $\{e,e,e'\}=0$; both conditions are symmetric in $e,e'$, cf.~\cite[Lemma 3.9]{Loos}).
	A non-zero tripotent is primitive (resp.\ maximal) if it cannot be written as the sum of two non-zero orthogonal tripotents (resp.\ if it is not ortoghonal to any non-zero tripotent). The rank $\rk(e)$ of a tripotent $e$ is the maximal length (i.e., the number of terms)  of a sequence of pairwise orthogonal non-zero tripotents $(e_j)$ such that $e=\sum_j e_j$. 
	A frame is a maximal sequence of pairwise orthogonal non-zero tripotents. The length of such sequences is called the rank $r$ of $Z$.
\end{deff}

The vector spaces generated by any two frames are conjugate under $\Aut_0(Z)$, as well as any two maximal tripotents (cf.~\cite[Proposition 5.2 and Theorem 5.3]{Loos}). In particular, the rank is well defined.
If, in addition, $Z$ is simple (i.e., cannot be decomposed as the product of two non-trivial positive Hermitian Jordan triple systems), then any two frames are conjugate under $\Aut_0(V)$  (cf.~\cite[Theorem 5.9]{Loos}). 

\begin{prop}\label{prop:20}
	Let $Z$ be a positive Hermitian Jordan triple system. Then, the following hold:
	\begin{enumerate}
		\item[\textnormal{(1)}] the mapping $(x,y)\mapsto \Tr D(x,y)$ is a (non-degenerate) $\Aut(Z)$-invariant scalar product on $Z$;
		
		\item[\textnormal{(2)}] $D(x,y)^*=D(y,x)$ with respect to the scalar product in~\textnormal{(1)}, that is,
		\[
		\langle \{x,y,z\}\vert u \rangle=\langle x\vert \{y,z,u\}\rangle
		\]
		for every $x,y,z,u\in Z$;
		
		\item[\textnormal{(3)}]  if $e,e'$ are two orthogonal tripotents in $Z$, then $D(e,e)$ and $D(e',e')$ commute, and $e+e'$ is a tripotent;
		
		\item[\textnormal{(4)}] (`spectral decomposition') for every non-zero $x\in Z$ there are a unique strictly increasing family $(\lambda_j)_{j=1,\dots, k}$ of elements of $(0,+\infty)$, and a unique family $(e_j)_{j=1,\dots,k}$ of pairwise orthogonal non-zero tripotents such that
		\[
		x=\sum_{j=1}^k \lambda_j e_j;
		\]
		
		\item[\textnormal{(5)}] with the notation of~\textnormal{(4)}, $D(x,x)=\sum_{j=1}^k \lambda_j^2 D(e_j,e_j)$ and $\norm{D(x,x)}=\lambda_k^2$ with respect to the scalar product in~\textnormal{(1)};
		
		\item[\textnormal{(6)}] the circular convex open set $\Set{x\in Z\colon \norm{D(x,x)}<1}$ is a bounded symmetric domain with \v Silov boundary equal to the set of maximal tripotents.
	\end{enumerate}
\end{prop}

\begin{proof}
	(1)--(2) This is~\cite[Corollary 3.16]{Loos}.
	
	(3) Cf.~\cite[Lemma 3.9]{Loos}.
	
	(4) Cf.~\cite[Corollary 3.12]{Loos}.
	
	(5) Cf.~\cite[Theorem 3.17 and Theorem 6.5]{Loos}.
\end{proof}

\begin{nott}
	The converse of (6) of  Proposition~\ref{prop:20} holds, cf.~\cite[Theorem 4.1]{Loos}. Consequently, we shall endow $\C^n$ with the unique (positive Hermitian) Jordan triple system structure whose associated symmetric domain is $D$ (we still keep the notation introduced in Notation~\ref{not:1}). In addition, we shall assume that the natural scalar product of $\C^n$ is  $(x,y)\mapsto \Tr D(x,y)$. With this convention, the rank of a tripotent $e$ is $\abs{e}^2$.
\end{nott}

\begin{deff}
	A  complex Jordan algebra is a  complex commutative algebra $A$ such that
	\[
	x^2(xy)=x(x^2 y)
	\]
	for every $x,y\in A$.
		
	An $x\in A$ is said to be invertible if the endomorphism $P x\colon y \mapsto 2 x(x y)-x^2 y$ of $A$ is invertible, in which case $x^{-1}\coloneqq (Px)^{-1}x$ is the inverse of $x$.
\end{deff}

\begin{oss}
	Let $A$ be a  complex Jordan algebra with identity $e$. Then, $x\in A$ is invertible if and only if it is invertible in the (associative) algebra  $\C[e,x]$ generated by $x$ and $e$, in which case $x^{-1}$ coincides with the inverse of $x$ in  $\C[e,x]$. (Cf.~\cite[Proposition II.3.1]{FarautKoranyi2}.)
\end{oss}

\begin{prop}\label{prop:21}
	(`Peirce decomposition') Let $e$ be a tripotent in $\C^n$. Denote by $V_\alpha(e)$ the eigenspace of $D(e,e)$ corresponding to the eigenvalue $\alpha\in \C$. 
	Then, $\C^n=V_1(e)\oplus V_{1/2}(e) \oplus V_0(e)$ (orthogonal sum) and $\{ V_j(e), V_k(e), V_\ell(e) \}\subseteq V_{j-k+\ell}(e)$ for every $j,k,\ell\in \Set{0,1/2,1}$. 
	
	Furthermore, $V_1(e)$ is a  complex Jordan algebra with the product $(x,y)\mapsto \{x,e,y\}$, for which $e$ is the identity; the mapping $Q(e)\colon x \mapsto   \{e,x,e\}  $ is a conjugation on  $V_1(e)$.
	
	Finally, the orthogonal projectors of $\C^n$ onto $V_1(e)$, $V_{1/2}(e)$, and $V_0(e)$ are $Q(e)^2$, $2(D(e,e)-Q(e)^2)$, and  $B(e,e)=I-2 D(e,e)+Q(e)^2$, respectively.
\end{prop}

Cf.~\cite[Theorem 3.13]{Loos} for a proof of the previous result (cf.~\cite[p.~5]{MackeyMellon} for the last assertion).

\begin{deff}
	We shall define $V_1(e), V_{1/2}(e)$ and $V_0(e)$ as in Proposition~\ref{prop:21}. We shall say that $D$ is of tube type if $V_1(e)=\C^n$ for some (hence every) maximal tripotent $e$.
\end{deff}

\begin{prop}\label{prop:23}
	Let $e$ be a tripotent in $\C^n$, and define $A(e)\coloneqq \Set{z\in V_1(e)\colon z=\{e,z,e\}}$.
	Then, there is a unique holomorphic polynomial $\det$ on $V_1(e)$ such that, if $x\in A(e)$ and $x=\sum_j \lambda_j e_j$ is the spectral decomposition of $x$ (with the $e_j$ pairwise orthogonal non-zero tripotents in $A(e)$ and the $\lambda_j$ distinct real numbers), then $\det(x)= \prod_{j} \lambda_j^{\rk(e_j)}$. 
	In addition, the following hold:
	\begin{enumerate}
		\item[\textnormal{(1)}] $\det$ is irreducible if and only if $V_1(e)$ is simple;
		
		\item[\textnormal{(2)}] if $V_1(e)$ is simple, then  $\det$   is the unique holomorphic polynomial of degree $\rk(e)$ such that $\det (e)=1$ and  such that $\det(k z)=\chi(k) \det z$ for every $k\in \Aut_0(V_1(e))$, where $\chi$ is a character of   $\Aut_0(V_1(e))$;
		
		\item[\textnormal{(3)}]  if $(e_1,\dots, e_{\rk(e)})$ is a frame of tripotents in $A(e)$, then $\det\big(\sum_j \lambda_j e_j\big)=\prod_j \lambda_j$ for every $(\lambda_j)\in \C^r$;
		
		\item[\textnormal{(4)}]  if $V_1(e)$ is simple and  $(e_1,\dots, e_{\rk(e)})$ is a frame of tripotents in $V_1(e)$, then $\abs*{\det\big(\sum_j \lambda_j e_j\big)}=\prod_j \abs{\lambda_j}$ for every $(\lambda_j)\in \C^r$;
		
		\item[\textnormal{(5)}] if $V_1(e)$ is simple, then  ${\det}_\C P(x)=(\det x)^{ \dim_\R V_1(e)/\rk(e)}$ for every $x\in V_1(e)$, where $P(x)=\{x,Q(e)\,\cdot\,,x \}$   (note that $ \dim_\R V_1(e)/\rk(e)$ is an integer).
	\end{enumerate}
\end{prop}

\begin{proof}
	The first assertion follows from~\cite[Proposition II.2.1 and Theorem III.1.2]{FarautKoranyi2}. (1) then follows from~\cite[Lemma 2.3]{KoranyiPukanszki}, while (2) is a consequence of ~\cite[Theorem 2.1]{FarautKoranyi}. Next, 	(3) follows from~\cite[Theorem III.1.2]{FarautKoranyi2} when $(\lambda_j)\in \R^r$, and by holomorphy in the general case. For (4), Observe that $\abs{\det}$ is necessarily invariant under the compact group $\Aut_0(V_1(e))$ by (2), so that the assertion follows from (3) and the fact that any two frames of tripotents are conjugate under $\Aut_0(V_1(e))$ (cf.~\cite[Theorem 5.9]{Loos}).
	Finally, (5) follows from~\cite[Proposition III.4.2]{FarautKoranyi2} when $x\in A(e)$, since in this case ${\det}_\C P(x)={\det}_\R (P(x)\vert_{A(e)})$, and by holomorphy in the general case.
\end{proof}

\begin{ex}\label{ex:2}
	As an example, we shall now describe the Jordan triple system corresponding to the irreducible bounded symmetric domains of type $(I_{p,p})$ ($p\Meg 1$). Let $Z$ be the space of complex $p\times p$ matrices, endowed with the triple product defined by
	\[
	\{x,y,z\}=\frac{1}{2}(x y^* z+z y^* x)
	\] 
	for every $x,y,z\in Z$, where $y^*$ denotes the conjugate transpose of $y$. In this case (identifying $Z$ with the space of $\C$-linear mappings from $\C^p$ to $\C^p$), a tripotent is simply a  partial isometry, and a maximal tripotent is a unitary operator. The spectral decomposition of an element of $Z$ is then (essentially) its singular value decomposition, so that the spectral norm is simply the operator norm. In addition, the rank of $Z$ is $p$.
	
	If we let $e$ be the identity of $\C^p$, then $e$ is a maximal tripotent,  and $V_1(e)=Z$, so that the bounded symmetric domain associated with $Z$ is of tube type.
	Notice that an element of $V_{1}(e)$ is invertible if and only if it is invertible in the usual sense (with the same inverse), since the product of $V_{1}(e)$ (that is, $(x,y)\mapsto \frac 1 2 (x y+ y x)$) and the usual product coincide on the algebra generated by (e and) any element of $V_{1}(e)$.  
	Finally, $\det$ on $Z=V_1(e)$ is the usual (complex) determinant.
\end{ex}

\section{The Equality $H^2(\mi_\alpha)=L^2(\mi_\alpha)$ on General Domains}\label{sec:4}

The following result summarizes how one may extend the idea of the proof of the ``only if'' part of Theorem~\ref{teo:4} to the case of general domains of tube type.

\begin{prop}\label{prop:42}
	Assume that $D$ is of tube type and let $e$ be a maximal tripotent; endow $\C^n=V_1(e)$ with the structure of a Jordan algebra  with product $(z,z')\mapsto \{z,e,z'\}$, and let $\det$ be the corresponding determinant polynomial (cf.~Proposition~\ref{prop:23}). Up to a rotation, we may assume that $\overline z=\{e,z,e\}$ for every $z\in \C^n$. 
	Let $R$ be the $\C^n$-valued homogeneous polynomial of degree $r-1$ such that $z^{-1}=R(z)/\det(z)$ for every invertible $z$ (cf.~\cite[Proposition II.2.4]{FarautKoranyi2}).
	
	Assume that there is a basis $L_1,\dots, L_n$ of $(\C^n)'$ and, for every $j=1,\dots, n$, holomorphic polynomials $p_j,q_j$ on $\C^n$, with $q_j$ everywhere non-zero on $D$ and $\Hc^{m-1}$-almost everywhere non-zero on $\phi_1^{-1}(\alpha)$, such that $(L^\sharp_j\circ R)\, q_j= p_j\det $ on $\phi_1^{-1}(\alpha)$. Then, $H^2(\mi_\alpha)=L^2(\mi_\alpha)$.
\end{prop}

\begin{proof}
	Observe that $\overline \zeta=\zeta^{-1}$ for every $\zeta\in \bD$, thanks to~\cite[Proposition 5.5]{Loos}, so that
	\[
	\overline{L_j(\zeta)}=L_j^\sharp(\overline \zeta)=L_j^\sharp (\zeta^{-1})=\frac{L_j^\sharp(R(\zeta))}{\det(\zeta)} =\frac{p_j(\zeta)}{q_j(\zeta)}
	\]
	for every $\zeta\in\phi_1^{-1}(\alpha)$ such that $q_j(\zeta)\neq0$, hence for $\mi_\alpha$-almost every $\zeta$ by Theorem~\ref{teo:3}. 
	Now, it is clear that $\frac{p_j}{q_j(\rho\,\cdot\,)}\in H^2(\mi_\alpha)$ for every $\rho\to 1^-$ (cf., e.g.,~\cite[Remark 2.17]{BoundedClark}), and that $\frac{p_j}{q_j(\rho\,\cdot\,)}\to \frac{p_j}{q_j}=\overline{L_j}$  $\mi_\alpha$-almost everywhere. 
	In addition, by Lemma~\ref{lem:13} $\abs*{\frac{p_j}{q_j(\rho\,\cdot\,)}}\meg 2^{k_j}\abs*{\frac{p_j}{q_j}}=2^{k_j}\max_\bD \abs{L_j}$ $\mi_\alpha$-almost everywhere, where $k_j$ is the degree of $q_j$. 
	Then, the dominated convergence theorem shows that $\frac{p_j}{q_j(\rho\,\cdot\,)}\to  \overline{L_j}$ in $L^2(\mi_\alpha)$, so that $\overline{L_j}\in H^2(\mi_\alpha)$. By the arbitrariness of $j=1,\dots, n$,  using Lemma~\ref{lem:12} we then see that $H^2(\mi_\alpha)$ contains every polynomial (holomorphic or not), hence also $C(\bD)$ by the Stone--Weierstrass theorem, hence $L^2(\mi_\alpha)$ by density.
\end{proof}

\begin{cor}\label{cor:1}
	Assume that $D$ is of tube type and let $e$ be a maximal tripotent; endow $\C^n=V_1(e)$ with the structure of a Jordan algebra  with product $(z,z')\mapsto \{z,e,z'\}$, and let $\det$ be the corresponding determinant polynomial (cf.~Proposition~\ref{prop:23}). If $\phi=\frac{p \det }{q}$ for some coprime holomorphic polynomials $p,q$ with $q(z)\neq 0$ for every $z\in D$, then $H^2(\mi_\alpha)=L^2(\mi_\alpha)$.
\end{cor}

Notice that, since $q$ vanishes nowhere on $D$ and all the irreducible factors of $\det$ vanish at $0$ (by homogeneity), saying that $p$ and $q$ are coprime is the same as saying that $p\det$ and $q$ are coprime. In particular, the zero set of $q$ in $\bD$ has (real) dimension $\meg m-2$ by Theorem~\ref{teo:3}.

\begin{proof}
	It suffices to observe that, if $L\in (\C^n)'$, then clearly
	\[
	(L^\sharp\circ R) (-\alpha q)= (L^\sharp\circ R) p \det 
	\]
	on  $\phi_1^{-1}(\alpha)$, so that the conclusion follows from Proposition~\ref{prop:42}.
\end{proof}

We shall present below (cf.~Proposition~\ref{prop:40}) some sufficient conditions for  the inequality $H^2(\mi_\alpha)\neq L^2(\mi_\alpha)$   which allow to show that this inequality may happen even if $D$ is an irreducible tube domain (cf.~Example~\ref{ex:1}).
We shall  prepare the statement of Proposition~\ref{prop:40} with some remarks on the manifolds of tripotents.

Let $M$ be the set of tripotents, so that $M$ is a compact (real) analytic submanifold of $\C^n$   (cf.~\cite[Theorem 5.6]{Loos}). Observe that, for every $j=1,\dots, r$, the set $M_{j}$ of tripotents of rank $j$ is a (real) algebraic set defined by the polynomials $z \mapsto \{z,z,z\}-z$ and $z \mapsto \abs{z}^2-j$; since $M_{j}$ is open an compact in $M$ (because the relations $\abs{z}^2=j$ and $j-1<\abs{z}^2<j+1$ are equivalent on $M$), it is clear that $M_{j}$ is a union of connected components of $M$, hence a (real) analytic submanifold of $\C^n$. 
Observe that $M_{j}$ is connected if $D$ is irreducible, since $K$ acts transitively on $M_{j}$ in this case (cf.~\cite[Corollary 5.12]{Loos}).  If, otherwise, $D$ is reducible and $D_1,\dots, D_\ell$ are its irreducible factors (so that $D=\prod_{k=1}^\ell D_k$ up to a rotation), then $M_j=M_j(D)$ is the disjoint union of the compact connected sets $\sum_{k=1}^\ell M_{j_k}(D_k)$, with $j_1+\cdots+j_\ell=j$, and $K=K(D)=\prod_{j=1}^\ell K(D_j)$ acts transitively on each of these sets (cf.~\cite[Theorem 5.9]{Loos}). Then, these sets are the connected components of $M_j$. We have therefore proved that $K$ acts transitively on every connected component of $M_j$.

Now, tale $e\in M_{r-1}$, and observe that $V_0(e)$ is a  positive Hermitian Jordan triple system of rank $1$ (cf.~Proposition~\ref{prop:21}), hence  one-dimensional (cf.~\cite[4.14--18]{Loos}). 
In addition, the symmetric domain defined by $V_0(e)$, that is, $V_0(e)\cap D$ (cf.~\cite[Theorem 6.3]{Loos}) is the disc of centre $0$ and radius $\abs{e}=\sqrt{\rk(e)}=1$ in $V_0(e)$.
 Observe, in addition, that the boundary component  associated with $e$ is $e+V_0(e)\cap D$, and that its relative boundary (that is, its boundary in $e+ V_0(e)$) is contained in $\bD$ (cf.~\cite[Theorem 6.3]{Loos}).

Next, consider the set $F\coloneqq \Set{(e,z)\in \C^n\times \C^n\colon e\in M_{r-1}, z\in e+ V_0(e)\cap D} $, and observe that $F$ is a (real) algebraic set defined by the polynomials $(e,z)\mapsto \{e,e,e\}-e $, $(e,z)\mapsto \abs{e}^2-(r-1)$, $(e,z)\mapsto (I-B(e,e))(z-e)$, and $(e,z)\mapsto \abs{z-e}^2-1$ (recall that $B(e,e)=I-2 D(e,e)+Q(e)^2$ is the orthogonal projector of $\C^n$ onto $V_0(e)$ by Proposition~\ref{prop:21}). 
In addition, arguing as before  it is readily seen that $K$ acts transitively on each connected component of $F$ (the action being $k\cdot (e,z)=(k e, k z)$), so that $F$ is also an analytic submanifold of $\C^n\times \C^n$. 
In addition, the mapping $\pr_2\colon F\ni (e,z)\mapsto z\in \bD$ is onto and a submersion. Indeed, the first fact follows from the previous discussion, since every maximal tripotent may be split as a sum of some $e'\in M_{r-1}$ and a primitive tripotent (and therefore belongs to the relative boundary of the boundary component passing through $e'$). 
This implies that $\pr_2$ induces a submersion on a non-empty Zariski open subset of $F$, hence everywhere since $K$ acts transitively on each connected component of $F$.

\begin{deff}\label{def:10}
	We say that a (real) algebraic hypersurface $V$ in $\bD$ contains a bundle of tori if the (real) dimension of the closed semi-algebraic set $\pr_2(\pr_1^{-1}(W)\cap F)$ is $m-1$, where $W$ denotes the closed semi-algebraic set of $e\in M_{r-1}$ such that $\pr_1^{-1}(e)\cap F\subseteq \pr_2^{-1}(V)\cap F$ (cf.~\cite[Corollary 1.18 and Theorem 1.1]{Coste} for the semi-algebraicity of $W$ and $\pr_2(\pr_1^{-1}(W)\cap F)$).
\end{deff}

In other words, $V$ contains a bundle of tori if it contains a full-dimensional set which is covered by relative boundaries of the boundary components of $D$ of (complex) dimension $1$.

\begin{prop}\label{prop:40}
	Take $\alpha\in \T$ and assume that $\Cl(\phi_1^{-1}(\alpha))$ contains a bundle of tori. Then, $L^2(\mi_\alpha)\neq H^2(\mi_\alpha)$.
\end{prop}

When $D$ is a polydisc, this amounts to the `only if' part of Theorem~\ref{teo:4}.

\begin{proof}
	Set $V\coloneqq\Cl(\phi_1^{-1}(\alpha))$, and let $W$ be the closed semi-algebraic set of $e\in M_{r-1}$ such that $\pr_1^{-1}(e)\cap F\subseteq \pr_2^{-1}(V)\cap F$. Since the semi-algebraic set $\pr_2(\pr_1^{-1}(W)\cap F)$ has dimension $m-1$, we may find a non-singular element $(e_0,z_0)$ of $\pr_1^{-1}(W)\cap F$ such that also $z_0$ is non-singular in $V$;  we may also assume that $\pr_2\colon W\to V$ is a submersion at $(e_0,z_0)$.
	Then, we may find an $(m-2)$-dimensional vector subspace $X$ of the tangent space $T_{e_0} W$ of $W$ at $e_0$ such that $\pr_2$ maps $\pr_1^{-1}(X)\cap T_{(e_0,z_0)} F$ onto $T_{z_0} V$. 
	Hence, we may find a connected $(m-2)$-dimensional analytic submanifold $M$ of $W$ with tangent space $X$ at $e_0$, so that $M'\coloneqq\pr_1^{-1}(M)\cap F$ is an $(m-1)$-dimensional analytic submanifold of $F$.
	By construction, the restriction $f$ of $\pr_2$ to $M'$ has rank $m-1$ at $(e_0,\zeta_0)$, hence in a neighbourhood of $(e_0,\zeta_0)$. 
	
	Let $L$ be a $\C$-linear functional on $\C^n$ which does not vanish identically on  $\C(\zeta_0-e_0)=V_0(e_0) $. 
	Up to shrinking $M'$ if necessary, we may then assume that $L$ does not vanish identically on $\C(\zeta-e)$ for every $(e,\zeta)\in M'$. Notice that, if $(e,\zeta)\in M'$, then the torus $e+\T(\zeta-e)$ is $\pr_2(\pr_1^{-1}(e)\cap F)$, hence is contained in $V$ since $M\subseteq W$.
	Therefore, $L$ induces a non-constant holomorphic polynomial of the boundary component $e+ \Db(\zeta-e)$ of $D$, so that $\overline L \not \in H^s(e+\T(\zeta-e))$ for every $s>0$, where $H^s(e+\T(\zeta-e))$ denotes the boundary values space of the Hardy space $H^s(e+\Db(\zeta-e))$.
	
	Assume by contradiction that $\overline L  \in H^2(\mi_\alpha)$. Then, there is a sequence $(p_k)$ of holomorphic polynomials on $\C^n$ which converges to $\overline L$ in $L^2(\mi_\alpha)$. In particular, by the area formula,
	\[
	\int_{f(M')} \frac{\abs{p_k(\zeta)-\overline{L(\zeta)}  }^2 }{ \abs{\nabla\phi(\zeta)}}\,\dd \Hc^{m-1}(\zeta)\Meg \int_{M'} \frac{\abs{p_k(\zeta)-\overline{L(\zeta)}}^2 }{ \abs{\nabla\phi(\zeta)}} \abs{\det f'(e,\zeta)}\,\dd \Hc^{m-1}(e,\zeta)
	\]
	for every $k\in\N$, where the $\Meg $ accounts for the possible non-injectivity of $f$.
	Applying the coarea formula associated with the mapping $\pr_1\colon M'\to M$ (whose coarea factors are $1$ since $\pr_1$ induces a surjective quasi-isometry between the tangent spaces), we then infer that
	\[
	\int_{f(M')} \frac{\abs{p_k(\zeta)-\overline{L(\zeta)}  }^2 }{ \abs{\nabla\phi(\zeta)}}\,\dd \Hc^{m-1}(\zeta)\Meg  \int_M \int_{M'\cap \pr_1^{-1}(e)} \frac{\abs{p_k(\zeta)-\overline{L(\zeta)}  }^2 }{ \abs{\nabla\phi(\zeta)}} \abs{\det f'(e,\zeta)}\,\dd \Hc^{1}(e,\zeta)\,\dd \Hc^{m-2}(e).
	\]
	Consequently, up to a subsequence,  $\int_{M'\cap \pr_1^{-1}(e)} \frac{\abs{p_k(\zeta)-\overline{L(\zeta)}  }^2 }{ \abs{\nabla\phi(\zeta)}} \abs{\det f'(e,\zeta)}\,\dd \Hc^{1}(e,\zeta)\to 0$ for $\Hc^{m-2}$-almost every $e\in M$. 
	Now observe that, since $f$ is analytic and $\det f(\zeta_0,e_0)\neq 0$, for $\Hc^{m-2}$-almost every $e\in M$ the analytic function $\det f'(e,\,\cdot\,)$ does not vanish identically on the torus $ M'\cap \pr_1^{-1}(e)$. Hence,  it may vanish only at a finite number of points, so that $\abs{\det f'(e,\,\cdot\,)}^{-s}$ is integrable for sufficiently small $s>0$. Analogously, $\nabla\phi$ induces a somewhere defined rational function on the torus $M'\cap \pr_1^{-1}(e)$ for $\Hc^{m-2}$-almost every $e\in M$ (since $\nabla \phi$ is defined in $\bD\setminus N$ for some algebraic set $N$ of dimension $\meg m-2$ by Theorem~\ref{teo:3}), so that $\abs{\nabla \phi}^s$ is integrable on $M'\cap \pr_1^{-1}(e)$ for sufficiently small $s>0$ (and for $\Hc^{m-2}$-almost every $e\in M$). We may then take $e\in M$ and $s>0$ such that $p_k\to \overline L$ in $L^s(M'\cap \pr_1^{-1}(e),\Hc^1)$. Since the $p_k$ induce elements of $H^s(M'\cap \pr_1^{-1}(e))$ and since this latter space is closed in $L^s(M'\cap \pr_1^{-1}(e),\Hc^1)$ and does not contain $\overline L$ by the previous remarks, this leads to a contradiction.	 Therefore, $\overline L\not \in H^2(\mi_\alpha)$, whence the result.
\end{proof}

Here is a simple corollary of Proposition~\ref{prop:40}, which may be also proved using the same strategy as in the proof of~\cite[Proposition 5.6 (4)]{BoundedClark}.

\begin{cor}
	Assume that  $\C^n$ may be split as the direct sum of two Jordan triple systems $V_1$ and $V_2$, with $\partial_v \phi=0$ for every $v\in V_1$. Then  $H^2(\mi_\alpha)\neq L^2(\mi_\alpha)$ for every $\alpha\in \T$.
\end{cor}

By~\cite[Theorem 3.3]{KoranyiVagi}, the assumptions are met whenever $D$ is not of tube type.

\begin{proof}
	By assumption, $\phi=\psi\circ \pr$, where $\pr\colon \C^n\to V_2$ is the orthogonal projection and $\psi$ is some rational inner function on the symmetric domain associated with $V_2$.
	If $B$ is the relative boundary of a boundary component (of complex dimension $1$) of the symmetric domain associated with $V_1$, then $B+ \zeta_2$ is the relative boundary of a boundary component (of complex dimension $1$) of $D$ and is contained in $\phi_1^{-1}(\alpha)$ for every  $\zeta_2\in \psi_1^{-1}(\alpha)$. Since the union of these sets covers $\phi_1^{-1}(\alpha)$,  the result follows from Proposition~\ref{prop:40}.
\end{proof}

We now provide as example showing that Proposition~\ref{prop:40} is not vacuous also when $D$  is irreducible.

\begin{ex}\label{ex:1}
	Assume that $D$ is the domain of $2\times 2$ complex matrices with (operator) norm $<1$, so that $D$ is irreducible and of tube type, and $\bD=U(2)$ (cf.~Example~\ref{ex:2}). Consider the rational   function defined by
	\[
	\phi\left( \begin{matrix}
		a &b \\ c & d
	\end{matrix} \right)= \frac{ad-bc-d}{1-a}.
	\]
	Observe that $\phi$ is well defined on $D$, since $\abs{a}<1$ for every $\left( \begin{smallmatrix}
		a &b \\ c & d
	\end{smallmatrix} \right)\in D$, so that $q\left( \begin{smallmatrix}
		a &b \\ c & d
	\end{smallmatrix} \right)=1-a\neq 0$. In addition, if $\left( \begin{smallmatrix}
	a &b \\ c & d
	\end{smallmatrix} \right)\in  U(2)$, then there are $\gamma\in \T$ and $(a',b')\in \partial B_{\C^2}((0,0),1)$ such that $a=\gamma a'$, $b=\gamma b'$, $c=-\gamma\overline{b'}$, and $d=\gamma \overline{a'}$, so that $\phi\left( \begin{smallmatrix}
	a &b \\ c & d
	\end{smallmatrix} \right)= \frac{ \gamma^2  - \gamma \overline{a'}  }{1- \gamma a'}\in \T $  (unless $a'\gamma=1$, in which case $\left( \begin{smallmatrix}
	a &b \\ c & d
	\end{smallmatrix} \right)=\left( \begin{smallmatrix}
	1 &0 \\ 0 & \gamma^2
	\end{smallmatrix} \right)$). 
	Applying Lemma~\ref{lem:13} as in the proof of Theorem~\ref{teo:3} and using the fact that $U(2)$ is the \v Silov boundary of $D$, we then see that $\abs{\phi}\meg 2$ on $D$. Since $\phi$ attains unimodular values on $\T$ (except on a $1$-dimensional set), we may then conclude that $\phi$ is inner.

	Let us now prove that $H^2(\mi_\alpha)\neq L^2(\mi_\alpha)$ for every $\alpha\in \T$. 
	To this aim, observe that for every $(x_1,x_2), (y_1,y_2)\in \partial B_{\C^2}((0,0),1)$ the linear mapping $L_{(x_1,x_2),(y_1,y_2)}$ such that $L_{(x_1,x_2),(y_1,y_2)}(x_1,x_2)=(y_1,y_2)$ and $L_{(x_1,x_2),(y_1,y_2)}(-\overline{x_2},\overline{x_1})=0$ is a tripotent of rank $1$. 
	Notice that this parametrization of tripotents of rank $1$ is redundant, since if we multiply both $(x_1,x_2)$ and $(y_1,y_2)$ by the same element of $\T$ we get the same linear mapping. 
	The relative boundary of the boundary component of $D$ passing through $L_{(x_1,x_2),(y_1,y_2)}$ is formed by the linear mappings $L_{(x_1,x_2),(y_1,y_2),\gamma}$, $\gamma\in \T$, such that $L_{(x_1,x_2),(y_1,y_2),\gamma}(x_1,x_2)=(y_1,y_2)$ and $L_{(x_1,x_2),(y_1,y_2),\gamma}(-\overline{x_2},\overline{x_1})=\gamma(-\overline{y_2},\overline{y_1})$, whose matrix is
	\[
	\left(\begin{matrix}
		y_1 & -\gamma \overline{y_2} \\ y_2 & \gamma \overline{y_1}
	\end{matrix}  \right) \left(\begin{matrix}
		\overline{x_1} & \overline{x_2}\\ -x_2 & x_1 
	\end{matrix}\right)=\left(\begin{matrix}
		\overline{x_1} y_1+\gamma x_2 \overline{y_2} & \overline{x_2} y_1- \gamma x_1 \overline{y_2}\\
		\overline{x_1} y_2 - \gamma x_2\overline{y_1} & \overline{x_2} y_2+\gamma x_1 \overline{y_1}
	\end{matrix}\right),
	\]
	and whose determinant is $\gamma$. Thus, $L_{(x_1,x_2),(y_1,y_2),\gamma}\in V_\alpha$, with the notation of Theorem~\ref{teo:3},  if and only if 
	\[
	\gamma-(\overline{x_2} y_2+\gamma x_1 \overline{y_1})=\alpha(1-  \overline{x_1} y_1-\gamma x_2 \overline{y_2}),
	\]
	that is,
	\[
	\gamma(1-x_1 \overline{y_1}+\alpha x_2 \overline{y_2} )=\overline{x_2} y_2+\alpha-\alpha \overline{x_1} y_1.
	\]
	This happens for every $\gamma\in \T$ if and only if
	\[
	1-x_1 \overline{y_1}+\alpha x_2 \overline{y_2}=0
	\]
	and
	\[
	\overline{x_2} y_2+\alpha-\alpha \overline{x_1} y_1=0,
	\]
	which are equivalent. This is the case, for example, if $\abs{x_1}<1$, $x_2=\sqrt{1-\abs{x_1}^2}$, $y_1=x_1$, and $y_2=-\alpha x_2=-\alpha \sqrt{1-\abs{x_1}^2}$. In this case,
	\[
	f(x_1,\gamma)\coloneqq L_{(x_1,x_2),(y_1,y_2),\gamma}=\left(\begin{matrix}
		\abs{x_1}^2( 1+\overline \alpha \gamma )-\overline\alpha\gamma & x_1\sqrt{1-\abs{x_1}^2} ( 1+\overline \alpha \gamma )\\
		-\overline{x_1}\sqrt{1-\abs{x_1}^2} ( \alpha+\gamma ) &  \abs{x_1}^2(\gamma+\alpha)- \alpha 
	\end{matrix}\right).
	\]
	We wish to show that $f\colon \Db\times \T\to U(2)$ has rank $3$ somewhere (hence in a non-empty open set). To this aim, it is convenient to consider the mapping $g\colon (0,1)\times (0,2\pi)\times (0,2\pi)\ni(\rho,\theta, \phi)\mapsto f(\rho \ee^{i\theta}, \ee^{i\phi})$, so that
	\[
	g(\rho,\theta,\phi)=\left(\begin{matrix}
		\rho^2( 1+\overline \alpha \ee^{i \phi} )-\overline\alpha \ee^{i \phi} & \rho\sqrt{1-\rho^2} \ee^{i \theta} ( 1+\overline \alpha \ee^{i \phi})\\
		-\rho\sqrt{1-\rho^2} \ee^{-i\theta} ( \alpha+\ee^{i \phi}) & \rho^2(\ee^{i \phi}+\alpha)- \alpha 
	\end{matrix}\right).
	\]
	Then, the Jacobian determinant of the real part of the $(1,2)$-component of $g$ and of the real  and imaginary parts of the $(2,2)$-component of $g$ is
	\[
	\begin{split}
		&\det\left(\begin{matrix}
			\frac{1-2 \rho^2}{\sqrt{1-\rho^2}}\Re(\ee^{i\theta}+\overline \alpha \ee^{i(\theta+ \phi)}) & -\rho\sqrt{1-\rho^2}\Im(\ee^{i\theta}+\overline \alpha \ee^{i(\theta+ \phi)}) & -\rho\sqrt{1-\rho^2}\Im(\overline \alpha \ee^{i(\theta+ \phi)}) \\
			2\rho (\cos \phi+\Re \alpha) & 0 & -\rho^2 \sin \phi\\
			2\rho (\sin \phi+\Im \alpha) & 0 & \rho^2 \cos\phi
		\end{matrix}\right)\\
		&\qquad= 2\rho^4\sqrt{1-\rho^2}\Im(\ee^{i\theta}+\overline \alpha \ee^{i(\theta+ \phi)})  (1+\Re \overline \alpha \ee^{i \phi}),
	\end{split}
	\]
	which is non-zero for $\ee^{i\phi}\neq -\alpha$ and $\ee^{-i\theta}\neq \pm \frac{1+\overline \alpha \ee^{i \phi}}{\abs{1+\overline \alpha \ee^{i \phi}}}$ (this latter condition may be removed replacing the real part of the $(1,2)$-component of $g$ with the imaginary part of the same component, if we are only interested in proving that $g$ has rank $3$). 
	Consequently, $V_\alpha$  contains a bundle of tori in the sense of Definition~\ref{def:10}, so that $H^2(\mi_\alpha)\neq L^2(\mi_\alpha)$ for every $\alpha\in \T$, thanks to Proposition~\ref{prop:40}.
\end{ex}

\end{document}